\documentclass[10pt]{amsart}

\usepackage[T1]{fontenc}
\usepackage[margin=1in]{geometry}
\usepackage{amsthm}
\usepackage{amsmath}
\usepackage{amssymb}
\usepackage{tikz-cd}
\usepackage{enumitem}
\usepackage{scalerel}
\usepackage[hidelinks]{hyperref}

\theoremstyle{plain}
\newtheorem{thm}{Theorem}[section]
\newtheorem{lem}[thm]{Lemma}
\newtheorem{cor}[thm]{Corollary}
\newtheorem{prop}[thm]{Proposition}
\newtheorem{mainthm}{Theorem}

\theoremstyle{definition}

\newtheorem{defn}[thm]{Definition}
\newtheorem{ex}[thm]{Example}
\newtheorem*{rem}{Remark}

\DeclareMathAlphabet\mathbfcal{OMS}{cmsy}{b}{n}

\renewcommand{\sfdefault}{iwona}
\DeclareMathAlphabet{\mathbfsf}{\encodingdefault}{\sfdefault}{bx}{n}
\DeclareMathOperator*{\colim}{colim}
\DeclareMathOperator*{\hocolim}{hocolim}

\let\oldsimeq\simeq
\renewcommand{\simeq}{\mathrel{\scaleobj{1.1}{\oldsimeq}}}

\newcommand{\Hom}{\mathrm{Hom}}
\DeclareMathOperator{\uHom}{\,\underline{\!\Hom\!}\,}

\newcommand{\op}{^{\mathrm{op}}}
\newcommand{\st}{\mathrm{st}}

\newcommand{\cl}{\mathrm{cl}}

\newcommand{\Mod}{\mathbfsf{Mod}}

\newcommand{\Pos}{\mathbfsf{Pos}}

\newcommand{\Fun}{\mathbfsf{Fun}}

\newcommand{\Sh}{\mathbfsf{Sh}}
\newcommand{\CoSh}{\mathbfsf{CoSh}}

\newcommand{\Ch}{\mathbfsf{Ch}}

\newcommand{\Asm}{\mathrm{Asm}}

\newcommand{\ev}{\mathrm{ev}}

\newcommand{\lc}{\mathrm{lc}}
\newcommand{\gc}{\mathrm{gc}}

\newcommand{\hlc}{\mathrm{hlc}}
\newcommand{\Ind}{\mathrm{Ind}}

\newcommand{\Ho}{\mathrm{Ho}}

\newcommand{\cofib}{\mathrm{cofib}}

\newcommand{\Lan}{\mathrm{Lan}}

\newcommand{\id}{\mathrm{id}}

\newcommand{\bbL}{\mathbb{L}}

\newcommand{\bbZ}{\mathbb{Z}}

\newcommand{\cA}{\mathcal{A}}
\newcommand{\cB}{\mathcal{B}}

\newcommand{\cE}{\mathcal{E}}

\newcommand{\cM}{\mathcal{M}}
\newcommand{\cN}{\mathcal{N}}

\newcommand{\cP}{\mathcal{P}}

\newcommand{\cX}{\mathcal{X}}
\newcommand{\cY}{\mathcal{Y}}

\newcommand{\cbB}{\mathbfcal{B}}

\newcommand{\hF}{\widehat{F}}
\newcommand{\hG}{\widehat{G}}

\newcommand{\hx}{\widehat{x}}
\newcommand{\hy}{\widehat{y}}

\newcommand{\sP}{{\mathsf P}}

\newcommand{\shF}{{\mathsf \hF}}
\newcommand{\shG}{{\mathsf \hG}}

\title[Universal Assembly and Cellular Loop Spaces on Regular CW Complexes]{Universal Assembly and Cellular Loop Spaces\\on Regular CW Complexes}
\author{S. Dylda, T. Macko}



\address{Faculty of Mathematics, Physics, and Informatics, Comenius University, Mlynsk\'a dolina,
SK-824 48, Bratislava, Slovakia} \email{tibor.macko@fmph.uniba.sk} \email{dylda1@uniba.sk} 

\address{Institute of Mathematics, Slovak Academy of Sciences, \v Stef\'anikova 49, Bratislava, SK-81473, Slovakia} \email{macko@mat.savba.sk}

\thanks{This work was supported by grants VEGA 1/0425/25, UK 1162/2024, UK UK/304/2023. Parts of this work will be included in the PhD thesis of S.D} 

\date{\today}

\begin{document}

\begin{abstract}
We develop a regular CW analogue of the classical assembly formalism for chain complexes appearing in algebraic surgery theory. From the cell poset, we construct combinatorial path and loop objects using fences of comparable cells and prove that their classifying spaces recover the homotopy types of the ordinary based path and loop spaces. The resulting loop object carries a natural monoid structure, giving rise to a DG algebra defined directly from the cellular structure.

For complexes of cellular cosheaves, we introduce a universal assembly functor to modules over the group ring of the fundamental group and study the localization determined by global equivalences. The associated homotopy category is identified with a Verdier quotient of the derived category of cellular cosheaves, and its fibrant objects are precisely the homotopy locally constant complexes. A single elementary cosheaf becomes a compact generator after localization, and its derived endomorphism DG algebra is identified with singular chains on the cellular loop space. Consequently, the localized theory admits a Morita description in terms of DG modules over the loop DG algebra. The formalism provides a regular CW counterpart of the classical delta-set approach to assembly in algebraic surgery theory due to Ranicki and Weiss.
\end{abstract}

\maketitle

\tableofcontents

\newpage

\section{Introduction}

The topological applications of algebraic $L$-theory involve a local-to-global procedure that assembles algebraic objects carried over a space into algebra over its fundamental group. This passage is used to encode geometric problems, such as the comparison of normal maps, surgery obstructions, and structure sets, in terms of chain level algebra equipped with the so-called \textit{chain duality}. This program, together with many applications, is developed in detail in the book \cite{Ranicki92} by Ranicki. The local objects reflect the fact that the geometry is built piece by piece, while the assembled object records how these local pieces are glued in the presence of a nontrivial fundamental group.

A key ingredient of \cite{Ranicki92} is laid out in the paper \cite{RanickiWeiss} by Ranicki and Weiss, where a theory of chain complexes over delta-sets is first studied without the chain duality. One starts with a local system (alias ``(co)sheaf'') of chain complexes over a delta-set $X$, performs assembly, and obtains a chain complex over the group ring $R[\pi_1(X,x)]$. In this formalism, the assembly functor is factored through a suitable localized category. In \cite{Ranicki92} the quadratic $L$-groups of the localized category and the assembled category were shown to be isomorphic; this result is known as the \textit{algebraic $\pi$-$\pi$ theorem} \cite[\S 10]{Ranicki92}. For its proof, it is important that in \cite{RanickiWeiss} the localized category is identified through the endomorphisms of a distinguished generator. In particular, a delta-set model of the based loop space of $X$ enters as the object controlling these endomorphisms.

The purpose of the present paper is to carry out an analogous program for regular CW complexes. Regular CW complexes arise naturally in geometric topology, and one often wants to work with the given cell structure rather than introduce an auxiliary delta-set model. This becomes especially important in multiplicative questions. After the full theory is developed, as in \cite{Ranicki92}, we expect the product formulas for surgery-theoretic invariants, such as the total surgery obstruction, to acquire straightforward proofs in the regular CW setting because the cartesian product of regular CW complexes carries a simple canonical product cell structure, see also \cite{MackoSpiros}. This happens to be in contrast with the more elaborate delta-set models described in \cite[Appendix B]{Ranicki92}, which require diagonal approximations. A direct regular CW formalism therefore avoids auxiliary choices and is better adapted to product constructions. 

In a regular CW complex, the local geometric pieces are the cells themselves, together with their incidence relations, so the natural way to record algebraic data over the space is to assign objects to cells and structure maps along face relations. This leads directly to functors on the cell poset $\cX$, and hence to the notion of a cellular sheaf or cosheaf. The use of sheaf-theoretic language is justified by results of Curry and Shepard; see \cite[Theorem~4.2.10]{Curry14}: for the Alexandrov topology associated with the cell poset, functor categories on $\cX$ are equivalent to categories of cellular sheaves and cosheaves on the corresponding topological space. This makes it natural to ask whether the assembly picture can be formulated entirely in terms of the cell poset, without passing through an auxiliary delta-set presentation. 

One expects the localized category to be controlled by chains on a based loop object. In the delta-set setting, such loop and path objects are already available through Kan's combinatorial constructions, see \cite[\S 9]{Kan58}. But in the regular CW setting, one would like a combinatorial model defined directly from the incidence data of cells. The present paper tackles this question by introducing a \textit{path poset} and a \textit{loop poset} associated with the cell poset. These posets are built from suitably comparable zig-zags in $\cX$, which are called \textit{fences}. Their classifying spaces recover the homotopy type of the corresponding based path and loop spaces, and they provide an explicit DG-algebra model for the derived endomorphisms of the generator of the localized category, whose objects we call \textit{homotopy local systems}. In this way, the regular CW assembly formalism acquires the same local-to-global shape as in \cite{RanickiWeiss}, but with the loop-space object described directly in cellular terms. 
\begin{center}
    \textbf{Notation}
    \vspace{0.2cm}
\end{center}

Throughout the discussion, we will employ the following notation:
\begin{itemize}
    \item $R$ - a commutative unital ring.
    \item $\Mod_R$ - category of modules over $R$.
    \item $\Ch(R)$ - category of unbounded chain complexes of R-modules.
    \item $D(R)$ - derived category of unbounded chain complexes of R-modules.
    \item $\Pos$ - category of (small) posets and order-preserving maps.
    \item $(X, \cX)$ - regular CW complex $X$ with a cell poset $\cX$.
    \item $\cbB C$ - classifying space of a small category $C$.
\end{itemize}

We write $\CoSh(\cX, R)$ for the abelian category of cellular cosheaves of $R$-modules on the cell poset $\cX$, and $D(\cX)$ for the derived category of complexes of such cosheaves. The precise definition of a \textit{cellular cosheaf} will be given in the later chapters.

\begin{center}
    \textbf{Main Theorems}
    \vspace{0.2cm}
\end{center}

Historically, the appearance of loop-space algebras was motivated by algebraic $K$-theory. In Waldhausen's $K$-theory, a connected based space is studied via modules over the loop-space ring spectrum $\Sigma^{\infty}_{+}\Omega X$. The loop DG algebra appearing in $L$-theory should be viewed as the chain-level linear analogue of a structure that was first exploited on the $K$-theoretic side \cite{Waldhausen85}. Hence, we first identify the loop-space object that controls the localized theory:

\begin{mainthm}[Cellular loop-space model]\label{thm:main-loop-space}
Let $(X,\cX)$ be a connected locally finite regular CW complex and let $x$ be a $0$-cell. There are posets $\cE_x\cX$ and $\varOmega_x\cX$, defined from fences in $\cX$, together with an endpoint map $\ev_1\colon \cE_x\cX\to \cX$ such that $\cbB\cE_x\cX$ is contractible and
\[
    \cbB\varOmega_x\cX \longrightarrow \cbB\cE_x\cX \xrightarrow{\;\cbB\ev_1\;} \cbB\cX
\]
is a homotopy fiber sequence. In particular,
\[
    \cbB\varOmega_x\cX \simeq \Omega_x X.
\]
Moreover, $\varOmega_x\cX$ is a monoid in the category $\Pos$ of small posets and order-preserving maps, so $C_*^{\mathrm{sing}}(\cbB\varOmega_x\cX;R)$ carries a natural DG-algebra structure.
\end{mainthm}

The main motivating point of considering the localized theory is the question of whether it is possible to suitably ``disassemble'' or ``fragment over $X$'' an $R[\pi]$-chain complex. In the delta-set setting, this was studied in \cite{RanickiWeiss}. It turns out that for the desired algebraic $\pi$-$\pi$-theorem, only the properties of certain specific complexes obtained in this way inside the localized category are needed; see \cite[\S 10, page 110]{Ranicki92}. The localization itself was described in \cite{RanickiWeiss} by explicit constructions via categories of fractions.

Our next result gives the regular CW analogue of the localization. To describe homotopy-theoretic properties of cellular (co)sheaves we use the language of model categories and derived categories. We assume slight familiarity with the standard notions and constructions of the theory. For background and terminology, we refer the reader to the monographs of Hovey and Hirschhorn \cite{Hovey99,Hirschhorn02}. Other necessary definitions are presented in Sections 2 and 3.

\begin{mainthm}[Homotopy-local-system localization]\label{thm:main-localization}
Let $(X,\cX)$ be a connected finite regular CW complex and let the abelian category of chain complexes of cellular cosheaves
\[
    \cM:=\Ch(\CoSh(\cX, R))
\]
be equipped with the projective model structure. Then the left Bousfield localization $L_S\cM$ of $\cM$ at global equivalences exists, and its fibrant objects are precisely the homotopy locally constant complexes of cellular cosheaves. If $D_\gc(\cX)\subset D(\cX)$ denotes the full triangulated subcategory of globally contractible objects, then there is an equivalence of triangulated categories 
\[
    D_\hlc(\cX) \cong D(\cX)/D_\gc(\cX).
\]
where $D_\hlc(\cX) := \Ho(L_S\cM)$ is the homotopy category and $D(\cX)/D_\gc(\cX)$ is the Verdier quotient of triangulated categories. Moreover, for any $0$-cell $x$, the image of the elementary cosheaf $\{\hx\}_R$ is a compact generator of $D_\hlc(\cX)$.
\end{mainthm}
\begin{rem}
    Unlike \cite{RanickiWeiss}, where the \textit{local systems} name is used for arbitrary cellular cosheaves, we use the name \textit{(homotopy) local systems} or \textit{(homotopy) locally constant cosheaves} to denote those cosheaves whose structure maps are (homotopy) equivalences, which is closer to the standard conventions from the sheaf theory.
\end{rem}
Theorem~\ref{thm:main-localization} identifies the localized category as the quotient. Factorization of the assembly functor through this quotient yields a canonical functor from $D_\hlc(\cX)$ to $D(R[\pi])$. The next theorem makes the description of this functor easier by allowing us to present the localized category as the category of modules over the loop DG-algebra and prescribing the functor by its action on the generators. It is an analogue of \cite[Theorem 4.2]{RanickiWeiss}:

\begin{mainthm}[DG-algebra model for homotopy local systems]\label{thm:main-dg-model}
Let $(X,\cX)$ be a connected finite regular CW complex, let $x$ be a $0$-cell and $\pi = \pi_1(X,x)$ be the fundamental group. There is a canonical homotopy local system $E_x$ such that $E_x$ is the fibrant replacement of $\{\hx\}_R$ and
\[
    \uHom_{D_\hlc(\cX)}(\{\hx\}_R,\{\hx\}_R) \cong \uHom_{D(\cX)}(\{\hx\}_R, E_x) \cong  C_*^{\mathrm{sing}}(\cbB\varOmega_x\cX\op;R)
\]
in $D(R)$. Consequently the model category of homotopy local systems is Quillen equivalent to the model category of DG-modules over the loop DG algebra
\[
    C_*^{\mathrm{sing}}(\cbB\varOmega_x\cX\op;R).
\]
This DG algebra admits an augmentation
\[
    C_*^{\mathrm{sing}}(\cbB\varOmega_x\cX\op;R) \longrightarrow R[\pi]
\]
which induces an isomorphism on $H_0$.
\end{mainthm}

\begin{rem}
    In the notation of \cite[\S 10]{Ranicki92} the system $E_x$ is an analogue of the \textit{homogeneous envelope} $V^\infty \Gamma$. 
\end{rem}

Taken together, these results show that the part of the homotopy theory of cellular cosheaves detected by localization is governed by a single generator whose endomorphisms are computed by chains on a cellular model of the based loop space. The augmentation to $R[\pi]$ recovers the usual passage from loop-space algebra to the group ring, and hence places the construction in the expected surgery-theoretic framework.\\

This paper develops a part of the technical framework for algebraic surgery over regular CW complexes. In the followup paper, we develop the corresponding chain duality, which allows us to construct $L$-theory in this setting. The results of the present paper are then used to prove the analogue of the algebraic $\pi$-$\pi$ theorem. In this way, we obtain the complete algebraic surgery exact sequence for finite regular CW complexes, see \cite[\S 14]{Ranicki92}, as well as a significant simplification of definitions and proofs of multiplicative formulas for surgery obstruction invariants, such as \cite[Proposition 21.1]{Ranicki92}. On the $L$-homology term of the sequence such formulas were already obtained in \cite{MackoSpiros}. \\

It should be mentioned, that there exists a treatment of the present theory for finite polyhedra in the setup of stable $\infty$-categories, presented mostly in \cite{LurieAlgSurg} and partially in \cite{9Authors}. The present paper attempts to achieve the middle ground between Ranicki's explicit approach and such abstract $\infty$-categorical treatment, where we both generalize to regular CW-complexes and utilize model-categorical tools to enjoy the advantages of abstract $1$-categorical reasoning, while retaining connections and freedom to use explicit formulas. \\

\noindent Here we will outline the structure of our paper:
\begin{itemize}
    \item Section~\ref{sec:cellular-loop-spaces} constructs the fence-based path $\cE_x\cX$ and loop posets $\varOmega_x\cX$ and compares their classifying spaces with the ordinary based path and loop spaces.
    \item Section~\ref{sec:cellular-cosheaves-and-assembly} introduces the basic formalism of cellular cosheaves on regular CW complexes, including elementary objects, the projective model structure, derived pushforwards. Universal assembly functor is also introduced in this section as well as the short recapitulation on local systems and their relation to algebra over the group ring $R[\pi]$.
    \item Section~\ref{sec:homotopy-local-systems} studies the localization determined by global equivalences, identifies its homotopy category $D_\hlc(\cX)$ with a Verdier quotient $D(\cX) / D_\gc(\cX))$, and proves that the elementary cosheaf at a base vertex becomes a compact generator.
    \item Section~\ref{sec:loop-dg-algebra} computes the derived endomorphism DG algebra of the generator, constructs the augmentation to the group ring, and deduces the Morita description of homotopy local systems.
\end{itemize}

Theorems A, B and C will be proved in the Sections \ref{sec:cellular-loop-spaces}, \ref{sec:homotopy-local-systems} and \ref{sec:loop-dg-algebra} respectively.

\section{Cellular loop spaces}\label{sec:cellular-loop-spaces}

The geometric input needed for Section~\ref{sec:loop-dg-algebra} is a poset-level model for the based path fibration. In ordinary topology, the endomorphisms of a based fiber are controlled by the based loop space; here the analogous role will be played by the generator coming from Section~\ref{sec:homotopy-local-systems}. For this section we only need the standing connected regular CW hypotheses together with local finiteness, and we denote the chosen base $0$-cell simply by $x$. The aim is to construct from the poset alone a combinatorial analogue of the based path fibration
\[
    \Omega_x X \longrightarrow \sP_x X \xrightarrow{\;\ev_1\;} X.
\]
The basic objects are zig-zags of comparable cells. They should be thought of as \textit{cellular} paths in a CW complex: one is allowed to move from a cell to a face or coface, and the order relation on such zig-zags is generated by elementary homotopies that refine, shift, or extend the zig-zag.

\subsection{Preliminaries}

Throughout, a regular CW complex means a CW complex in which each attaching map is a homeomorphism onto its image. If $(X,\cX)$ is a regular CW complex, we write $\cX$ for its cell poset, ordered by
\[
    x\le y \qquad \Longleftrightarrow \qquad x\subseteq \overline{y}.
\]
For a cell $x\in \cX$ we write
\[
    \st(x):=\{\,y\in \cX\mid x\le y\,\},
    \qquad
    \cl(x):=\{\,y\in \cX\mid y\le x\,\}.
\]
We use the classifying-space notation $\cbB\cA:=|\cN(\cA)|$ for any small category $\cA$.

The combinatorics of a regular CW complex are encoded faithfully by its cell poset. In particular, $\cbB\cX$ identifies with the barycentric subdivision of $X$, so there is a canonical homeomorphism
\[
    \cbB\cX \cong X.
\]
This is the main reason that functor categories on $\cX$ can be used as a genuinely cellular model for sheaf-theoretic constructions on $X$.

\subsection{Fences and path posets}

\begin{defn}
A \textbf{based fence at $x$} is a finite sequence
\[
    z=(x=x_0,x_1,\dots,x_n)
\]
of cells of $\cX$ such that each consecutive pair $x_i,x_{i+1}$ is comparable. When convenient we display the relevant order symbols between consecutive entries, for example
\[
    z=(x=x_0 \le x_1 \ge x_2 \le \cdots \ge x_n).
\]
The integer $n$ is called the \textbf{length} of the fence.
\end{defn}

Two consecutive equal entries carry no geometric information, so they are removed.

\begin{defn}
Let $\widetilde{\cE}_x$ be the set of based fences at $x$. Define an equivalence relation on $\widetilde{\cE}_x$ generated by deleting a repeated adjacent entry,
\[
    (x_0,\dots,x_i,x_i,\dots,x_n) \sim (x_0,\dots,x_i,\dots,x_n).
\]
A fence is \textbf{reduced} if no two consecutive entries coincide.
\end{defn}

We next define the order relation. Intuitively, a fence becomes \emph{smaller} when an ascending step is refined or an endpoint is pushed downward, and becomes \emph{larger} under the dual descending operations. Peaks may move upward and valleys may move downward.

\begin{defn}\label{defn:fence-order}
On $\widetilde{\cE}_x$ define a relation generated by the following elementary moves.
\begin{enumerate}[label=(\arabic*)]
    \item \textbf{Internal refinements.}
    \[
        (\cdots a \le b \cdots) \ge (\cdots a \le y \le b \cdots)
        \qquad\text{whenever } a\le y\le b,
    \]
    and dually
    \[
        (\cdots a \ge b \cdots) \le (\cdots a \ge y \ge b \cdots)
        \qquad\text{whenever } a\ge y\ge b.
    \]

    \item \textbf{Internal shifts.}
    \[
        (\cdots a \le b \ge c \cdots) \le (\cdots a \le y \ge c \cdots)
        \qquad\text{whenever } b\le y,
    \]
    and dually
    \[
        (\cdots a \ge b \le c \cdots) \ge (\cdots a \ge y \le c \cdots)
        \qquad\text{whenever } b\ge y.
    \]

    \item \textbf{Endpoint extension.}
    \[
        (\cdots a) \le (\cdots a \le y)
        \qquad\text{whenever } a\le y,
    \]
    and dually
    \[
        (\cdots a) \ge (\cdots a \ge y)
        \qquad\text{whenever } a\ge y.
    \]
\end{enumerate}
The \textbf{cellular path poset} $\cE_x\cX$ is the quotient of $\widetilde{\cE}_x$ equipped with the induced relation, given by the reflexive and transitive closure of the elementary relations above.
\end{defn}

The endpoint of a based fence depends only on its equivalence class, so there is a well-defined endpoint map
\[
    \ev_1\colon \cE_x\cX \longrightarrow \cX,
    \qquad
    \ev_1(x_0,\dots,x_n)=x_n.
\]
For each $y\in \cX$ define the fiber poset
\[
    \cP_{x,y}\cX:=\ev_1^{-1}(y)
    \subseteq \cE_x\cX,
    \qquad
    \varOmega_x\cX:=\cP_{x,x}.
\]
Thus $\cP_{x,y}\cX$ is the poset of cellular paths from $x$ to $y$, and $\varOmega_x\cX$ is the poset of cellular loops based at $x$.

\begin{prop}\label{prop:fence-order-poset}
The induced relation on $\cE_x\cX$ is a partial order.
\end{prop}

\begin{proof}
Reflexivity and transitivity are built into the definition, so only antisymmetry requires proof. Let $z,z'\in \cE_x\cX$ satisfy $z\le z'$ and $z'\le z$. Choose reduced representatives, still denoted $z$ and $z'$, having minimal possible length in their respective equivalence classes.

An elementary move of type \textup{(1)} or \textup{(3)} changes the length by one unless it inserts a repeated adjacent entry; the latter possibility is invisible in $\cE_x\cX$ because we have already passed to the degeneracy quotient. Consequently, no zig-zag of inequalities from $z$ to $z'$ and back can contain a genuine move of type \textup{(1)} or \textup{(3)}: such a move would strictly change the minimal length and could only be undone by passing through a degenerate fence, contrary to the choice of reduced representatives.

Hence every step in both chains is of type \textup{(2)}. Consider one such step,
\[
    (\cdots a \le b \ge c \cdots) \le (\cdots a \le y \ge c \cdots).
\]
If the reverse inequality also occurs along the chain back, then necessarily $b\le y$ and $y\le b$. Since $\cX$ is a poset, $b=y$, so the move is trivial. The same argument applies to downward shifts of valleys. Therefore every elementary step appearing in both directions is trivial, and we conclude $z=z'$. This proves antisymmetry.
\end{proof}

\begin{ex}[The interval]
Let $X=[0,1]$ with regular CW structure consisting of vertices $v_0,v_1$ and one edge $e$. If the basepoint is $x=v_0$, then every reduced fence from $v_0$ to $v_1$ is comparable to
\[
    (v_0 \le e \ge v_1).
\]
In particular $\cP_{v_0,v_1}$ is contractible, and the loop poset $\varOmega_{v_0}\cX$ is also contractible. This matches the fact that both the path space from $0$ to $1$ and the based loop space of an interval are contractible.
\end{ex}

\begin{ex}[A cellular circle]
Let $X=S^1$ with two vertices $v_0,v_1$ and two edges $e_+,e_-$ joining them. With basepoint $x=v_0$ there are two basic reduced loop fences
\[
    \ell_+=(v_0 \le e_+ \ge v_1 \le e_- \ge v_0),
    \qquad
    \ell_-=(v_0 \le e_- \ge v_1 \le e_+ \ge v_0).
\]
They represent the two directions around the circle, and concatenations of these fences give combinatorial representatives of all winding classes. Theorem~\ref{thm:loop-comparison} below identifies
\[
    \pi_0(\cbB\varOmega_{v_0}\cX) \cong \pi_1(S^1,v_0) \cong \bbZ.
\]
\end{ex}

\subsection{Path objects and loop fibers}

We begin by showing that $\cE_x\cX$ is contractible. The argument is completely combinatorial and uses a filtration by fence length.

\begin{prop}\label{prop:Ex-contractible}
For every cell $x\in \cX$, the classifying space $\cbB\cE_x\cX$ is contractible.
\end{prop}

\begin{proof}
For $n\ge 0$ let $\cE_x\cX^{(n)}\subseteq \cE_x\cX$ be the full subposet spanned by fences of length at most $n$. Then
\[
    \cE_x\cX^{(0)}=\{(x)\}
\]
and
\[
    \cE_x\cX=\colim_n \cE_x\cX^{(n)}
\]
as a filtered union.

For each $n\ge 0$ define a truncation map
\[
    p_n\colon \cE_x\cX^{(n+1)} \longrightarrow \cE_x\cX^{(n)}
\]
by deleting the last entry of a fence of length exactly $n+1$ and acting as the identity on shorter fences. We first verify that $p_n$ is order preserving. Because the order on $\cE_x\cX$ is generated by elementary moves, it is enough to check compatibility with those moves that involve the last step of a fence. The only possibilities are:
\begin{itemize}
    \item an ascending refinement at the end,
    \[
        (\cdots a \le b) \ge (\cdots a \le y \le b) \quad\text{truncates to}\quad (\cdots a \le b) \ge (\cdots a \le y) \quad  \text{because} \quad y \le b
    \]
    \item a descending refinement at the end,
    \[
        (\cdots a \ge b) \le (\cdots a \ge y \ge b) \quad\text{truncates to}\quad (\cdots a \ge b) \le (\cdots a \ge y)  \quad  \text{because} \quad y\ge b
    \]
    \item an endpoint extension,
    \[
        (\cdots a) \le (\cdots a \le y)
        \quad\text{or}\quad
        (\cdots a) \ge (\cdots a \ge y) \quad \text{becomes identity after truncation}
    \]
\end{itemize}
All other elementary moves take place away from the endpoint and are unaffected by truncation. Thus $p_n$ is order preserving.

For every $z\in \cE_x\cX^{(n+1)}$ one has either $p_n(z)\le z$ or $z\le p_n(z)$, according to whether the final step of $z$ is descending or ascending. Let
\[
    A_n:=\{\,z\in \cE_x\cX^{(n+1)} \mid p_n(z)\le z\,\}.
\]
Define
\[
    r_n\colon \cE_x\cX^{(n+1)} \longrightarrow A_n,
    \qquad
    r_n(z):=
    \begin{cases}
        z, & z\in A_n,\\
        p_n(z), & z\notin A_n.
    \end{cases}
\]
We claim that $r_n$ is order preserving. Suppose $z\le z'$. There are four cases.
\begin{itemize}
    \item If $z,z'\in A_n$, then $r_n(z)=z\le z'=r_n(z')$.
    \item If $z,z'\notin A_n$, then $r_n(z)=p_n(z)\le p_n(z')=r_n(z')$ because $p_n$ is order preserving.
    \item If $z\in A_n$ and $z'\notin A_n$, then $r_n(z)=z\le z'\le p_n(z')=r_n(z')$.
    \item If $z\notin A_n$ and $z'\in A_n$, then $r_n(z)=p_n(z)\le p_n(z') \le z'=r_n(z')$.
\end{itemize}
Hence $r_n$ is a functor. Let $j_n\colon A_n\hookrightarrow \cE_x\cX^{(n+1)}$ be the inclusion. Then $r_n j_n=\id_{A_n}$ and there is a natural transformation
\[
    \id_{\cE_x\cX^{(n+1)}} \Longrightarrow j_n r_n,
    \qquad
    z\longmapsto z\le r_n(z),
\]
where the inequality is equality on $A_n$ and is $z\le p_n(z)$ off $A_n$. Therefore $\cbB A_n\simeq \cbB \cE_x\cX^{(n+1)}$.

Now consider the inclusion $i_n\colon \cE_x\cX^{(n)}\hookrightarrow A_n$. Since $p_n i_n=\id_{\cE_x\cX^{(n)}}$ and every $z\in A_n$ satisfies $i_n p_n(z)=p_n(z)\le z$, we also obtain a natural transformation
\[
    i_n p_n \Longrightarrow \id_{A_n}.
\]
Thus $\cbB\cE_x\cX^{(n)}\simeq \cbB A_n$. Combining the two equivalences gives
\[
    \cbB\cE_x\cX^{(n)} \simeq \cbB\cE_x\cX^{(n+1)}
    \qquad\text{for all } n\ge 0.
\]
Since $\cbB\cE_x\cX^{(0)}\cong *$, every $\cbB\cE_x\cX^{(n)}$ is contractible.

Finally,
\[
    \cN(\cE_x\cX)=\colim_n \cN(\cE_x\cX^{(n)})
\]
because every simplex in the nerve uses only finitely many fences and therefore lies in some finite-length stage. The transition maps are monomorphisms of simplicial sets, hence cofibrations. A filtered colimit of contractible simplicial sets along cofibrations is contractible, so $\cN(\cE_x\cX)$ is contractible. Geometric realization therefore gives $\cbB\cE_x\cX\simeq *$.
\end{proof}

The next two results identify the under-categories of the endpoint map and compare path posets with different targets.

\begin{prop}\label{prop:comma-retract}
For every cell $y\in \cX$, let
\[
    (y\downarrow \ev_1):=\{\, z\in \cE_x\cX \mid y\le \ev_1(z)\,\}
\]
be the usual under-category, viewed here as a full subposet of $\cE_x\cX$. Then the inclusion
\[
    \iota_y\colon \cP_{x,y}\cX \hookrightarrow (y\downarrow \ev_1)
\]
exhibits $\cP_{x,y}\cX$ as a deformation retract of $(y\downarrow \ev_1)$. In particular,
\[
    \cbB(y\downarrow \ev_1) \simeq \cbB\cP_{x,y}\cX.
\]
\end{prop}

\begin{proof}
Define a retraction
\[
    s_y\colon (y\downarrow \ev_1) \longrightarrow \cP_{x,y}\cX
\]
by appending $y$ at the endpoint:
\[
    s_y(x=x_0,\dots,x_n):=(x=x_0,\dots,x_n \ge y).
\]
This is well defined because $y\le x_n=\ev_1(z)$ for every object of $(y\downarrow \ev_1)$. It is order preserving because each elementary move relating two fences remains valid after the same endpoint extension is appended to both sides.

For $z\in \cP_{x,y}\cX$ one has $s_y(z)=z$ in the quotient, since the appended final step is $(y\ge y)$ and is removed by degeneracy. Thus $s_y\iota_y=\id_{\cP_{x,y}\cX}$. On the other hand every $z\in (y\downarrow \ev_1)$ satisfies
\[
    s_y(z)\le z
\]
by the endpoint-extension relation, so there is a natural transformation
\[
    \iota_y s_y \Longrightarrow \id_{(y\downarrow \ev_1)}.
\]
Hence $\iota_y$ induces a homotopy equivalence on classifying spaces.
\end{proof}

\begin{prop}\label{prop:path-posets-comparable}
If $a,b\in \cX$ are comparable, then
\[
    \cbB\cP_{x,a}\cX \simeq \cbB\cP_{x,b}\cX.
\]
\end{prop}

\begin{proof}
Assume first that $a\le b$. Define functors
\[
    s_{a,b}\colon \cP_{x,a}\cX \longrightarrow \cP_{x,b}\cX,
    \qquad
    r_{b,a}\colon \cP_{x,b}\cX \longrightarrow \cP_{x,a}\cX
\]
by endpoint extension,
\[
    s_{a,b}(x= x_0,\dots,a):=(x=x_0,\dots,a\le b),
    \qquad
    r_{b,a}(x=x_0,\dots,b):=(x=x_0,\dots,b\ge a).
\]
Both are order preserving for the same reason as in Proposition~\ref{prop:comma-retract}.

For $z\in \cP_{x,a}\cX$ we have
\[
    r_{b,a}s_{a,b}(z)=(\cdots a\le b\ge a)\ge (\cdots a)=z,
\]
so there is a natural transformation
\[
    \id_{\cP_{x,a}\cX} \Longrightarrow r_{b,a}s_{a,b}.
\]
For $w\in \cP_{x,b}\cX$ we similarly have
\[
    s_{a,b}r_{b,a}(w)=(\cdots b\ge a\le b)\le (\cdots b)=w,
\]
which gives a natural transformation
\[
    s_{a,b}r_{b,a} \Longrightarrow \id_{\cP_{x,b}\cX}.
\]
Therefore $s_{a,b}$ and $r_{b,a}$ induce inverse homotopy equivalences on classifying spaces.

If instead $a\ge b$, the same construction with the order signs reversed gives natural transformations in the opposite directions. The conclusion is again that $\cbB\cP_{x,a}\cX$ and $\cbB\cP_{x,b}\cX$ are homotopy equivalent.
\end{proof}

\begin{cor}\label{cor:path-posets-all}
For every $y\in \cX$ there is a homotopy equivalence
\[
    \cbB\cP_{x,y}\cX \simeq \cbB\varOmega_x\cX.
\]
\end{cor}

\begin{proof}
Because $\cbB\cX \cong X$ is connected, the comparability graph of $\cX$ is connected. Hence there exists a fence
\[
    (y=s_0, s_1,\dots, s_m=x)
\]
from $y$ to $x$. Successive pairs $s_i,s_{i+1}$ are comparable, so repeated application of Proposition~\ref{prop:path-posets-comparable} yields
\[
    \cbB\cP_{x,y}\cX
    \simeq \cbB\cP_{x,s_1}\cX
    \simeq \cdots \simeq \cbB\cP_{x,s_m}\cX
    = \cbB\cP_{x,x}\cX
    = \cbB\varOmega_x\cX.
\]
\end{proof}

\subsection{Comparison with based loops}

We now identify the cellular path object with the usual topological one after passing to classifying spaces.

\begin{thm}\label{thm:loop-comparison}
For every $y\in \cX$ the classifying space $\cbB\cP_{x,y}\cX$ is naturally homotopy equivalent to the homotopy fiber of $\cbB\ev_1$, hence the sequence
\[
    \cbB\cP_{x,y}\cX \longrightarrow \cbB\cE_x\cX \xrightarrow{\;\cbB\ev_1\;} \cbB\cX
\]
is a homotopy fiber sequence.
\end{thm}

\begin{proof}
Fix a relation $a\le b$ in $\cX$. The induced functor between under-categories is the inclusion
\[
    (b\downarrow \ev_1) \hookrightarrow (a\downarrow \ev_1),
\]
because $b\le \ev_1(z)$ implies $a\le \ev_1(z)$. By Proposition~\ref{prop:comma-retract} and Proposition~\ref{prop:path-posets-comparable} we have a chain of homotopy equivalences
\[
    \cbB(b\downarrow \ev_1)
    \simeq \cbB\cP_{x,b}\cX
    \simeq \cbB\cP_{x,a}\cX
    \simeq \cbB(a\downarrow \ev_1).
\]
Thus every morphism of $\cX$ induces a homotopy equivalence on the corresponding under-categories of $\ev_1$.

Quillen's Theorem~B \cite{Quillen72} therefore applies to the functor $\ev_1\colon \cE_x\cX\to \cX$. It implies that for each object $y\in \cX$ the square
\[
    \begin{tikzcd}[row sep=large, column sep=large]
        \cbB(y\downarrow \ev_1) \arrow[r] \arrow[d] & \cbB\cE_x\cX \arrow[d, "\cbB\ev_1"] \\
        * \arrow[r, "y"] & \cbB\cX
    \end{tikzcd}
\]
is homotopy cartesian. By Proposition~\ref{prop:comma-retract}, the upper-left corner is homotopy equivalent to $\cbB\cP_{x,y}\cX$. This proves the claim.
\end{proof}

\begin{cor}\label{cor:thm-A-case}
    Statement of \ref{thm:loop-comparison} is the generalized version of Theorem A: specializing to $y=x$ gives that $\cbB\varOmega_x\cX$ is the homotopy fiber of $\cbB\ev_1$ over $x$. Since $\cbB\cE_x\cX$ is contractible by Proposition~\ref{prop:Ex-contractible}, the homotopy fiber is homotopy equivalent to the based loop space $\Omega_x\cbB\cX$. Finally $\cbB\cX\cong X$ for a regular CW complex, so $\cbB\varOmega_x\cX\simeq \Omega_x X$.
\end{cor}

\subsection{Monoid structure}

The cellular loop poset carries a strictly associative concatenation operation.

\begin{prop}\label{prop:loop-monoid}
Concatenation of fences defines an order-preserving map
\[
    *\colon \varOmega_x\cX \times \varOmega_x\cX \longrightarrow \varOmega_x\cX,
\]
\[
    (x=x_0,\dots,x_n=x) * (x=y_0,\dots,y_m=x)
    :=
    (x=x_0,\dots,x_n=x=y_0,\dots,y_m=x),
\]
and makes $\varOmega_x\cX$ a monoid object in $\Pos$.
\end{prop}

\begin{proof}
The concatenated sequence is again a loop fence based at $x$, because the last entry of the first fence and the first entry of the second fence are both $x$. Degeneracies remove the repeated middle occurrence of $x$, so the operation is well defined on equivalence classes.

To show monotonicity, let $z\le z'$ and $w\le w'$ in $\varOmega_x\cX$. Any elementary move relating $z$ to $z'$ takes place entirely inside the first factor, and any elementary move relating $w$ to $w'$ takes place entirely inside the second factor. Performing those same moves after concatenation shows
\[
    z*w \le z'*w'.
\]
The unit is the length-zero fence $(x)$. Indeed,
\[
    (x)*z=z=z*(x)
\]
after deleting the repeated middle entry $x$. Associativity is strict at the level of concatenation of sequences, hence also after passing to the quotient. Therefore $\varOmega_x\cX$ is a monoid in $\Pos$.
\end{proof}

\begin{cor}\label{cor:loop-dg-algebra}
The nerve $\cN(\varOmega_x\cX)$ is a simplicial monoid, the realization $\cbB\varOmega_x\cX$ is a topological monoid, and for every commutative ring $R$ the singular chains
\[
    C_*^{\mathrm{sing}}(\cbB\varOmega_x\cX;R)
\]
carry a natural DG-algebra structure.
\end{cor}

\begin{proof}
The nerve functor preserves products, so Proposition~\ref{prop:loop-monoid} makes $\cN(\varOmega_x\cX)$ into a simplicial monoid. Geometric realization preserves finite products of simplicial sets, hence $\cbB\varOmega_x\cX$ is a topological monoid. Applying singular chains and the Eilenberg--Zilber shuffle map yields an associative product on chains with unit induced by the constant loop.
\end{proof}

\begin{proof}[Proof of Theorem A]
    It follows from Corollary \ref{cor:thm-A-case} and Corollary \ref{cor:loop-dg-algebra}.
\end{proof}

\section{Cellular cosheaves and assembly}\label{sec:cellular-cosheaves-and-assembly}

\begin{defn}
    A \textbf{cellular cosheaf} of $R$-modules on $(X,\cX)$ is a functor
\[
    \shF\colon \cX\op \longrightarrow \Mod_R.
\]
Equivalently, one may view $\shF$ as a $\cX$-constructible cosheaf on $X$, or as a cosheaf on the Alexandrov topology associated to the cell poset $\cX$; by the Curry--Shepard correspondence, see \cite[Theorem~4.2.10]{Curry14}, these descriptions determine equivalent categories. The category of such objects is denoted by
\[
    \CoSh(\cX, R):=\Fun(\cX\op,\Mod_R).
\]
\end{defn}
Similarly one has cellular sheaves $\Sh(\cX,R)=\Fun(\cX,R)$, but after this point we only use sheaves when a comparison is needed. These categories are abelian (and even Grothendieck) categories.

For any $M\in \Mod_R$ we write $\underline{M}$ for the constant cosheaf with value $M$ and identity structure maps.

\subsection{Elementary objects and model structure}

For each cell $x\in \cX$ there is an evaluation functor
\[
    \ev_x\colon \CoSh(\cX, R) \longrightarrow \Mod_R,
    \qquad
    \ev_x(\shF)=\shF(x).
\]
Its left adjoint is the elementary cosheaf supported on the closure of $x$.

\begin{defn}
For $x\in \cX$ and $M\in \Mod_R$, the \textbf{elementary cosheaf}
\[
    \{\hx\}_M \in \CoSh(\cX, R)
\]
is defined by
\[
    \{\hx\}_M(y):=
    \begin{cases}
        M, & y\in \cl(x),\\
        0, & \text{otherwise,}
    \end{cases}
\]
with identity structure maps on the full subposet $\cl(x)$.
\end{defn}

\begin{prop}\label{prop:elementary-adjunction}
For every $x\in \cX$, $M\in \Mod_R$, and $\shF\in \CoSh(\cX, R)$ there is a natural isomorphism
\[
    \Hom_{\CoSh(\cX, R)}(\{\hx\}_M,\shF)
    \cong
    \Hom_{R}(M,\shF(x)).
\]
In particular, if $P$ is projective over $R$, then $\{\hx\}_P$ is projective in $\CoSh(\cX, R)$.
\end{prop}

\begin{proof}
This is the usual Kan-extension adjunction for evaluation at a point. The last statement follows because $\ev_x$ is exact.
\end{proof}

\noindent We use the following notation 
\[
    \cM:=\Ch(\CoSh(\cX, R))
\] 
for the abelian category of chain complexes of cellular cosheaves. It carries the standard projective model structure:

\begin{prop}\label{prop:projective-model-structure}
The category $\cM$ admits a cofibrantly generated projective model structure in which
\begin{itemize}
    \item weak equivalences are objectwise quasi-isomorphisms,
    \item fibrations are objectwise epimorphisms,
    \item cofibrations are monomorphisms with objectwise projective and K-projective cokernel.
\end{itemize}
We denote the corresponding homotopy category by
\[
    D(\cX):=\Ho(\cM).
\]
In particular, bounded-below complexes of coproducts of elementary projective cosheaves are cofibrant.
\end{prop}

\begin{proof}[Proof]
This is the projective model structure on complexes in a functor category. Since the elementary cosheaves $\{\hx\}_P$ are projective, bounded-below complexes built from them are K-projective and hence cofibrant.
\end{proof}

For finite $\cX$, bounded complexes of finitely generated projective cellular cosheaves are perfect, hence compact in $D(\cX)$. This is the class of objects relevant to the later generator statements.

\subsection{Derived pushforwards}

Let $f\colon \cX\to \cY$ be an order-preserving map of posets. Precomposition gives a pullback functor
\[
    f^*\colon \CoSh(\cY, R) \longrightarrow \CoSh(\cX, R),
    \qquad
    (f^*\shG)(x)=\shG(f(x)).
\]
Its left adjoint is the cellular pushforward, defined by left Kan extension.

\begin{defn}
For a cellular cosheaf $\shF\in \CoSh(\cX, R)$, define
\[
    f_*\shF := \Lan_{f\op}(\shF) \in \CoSh(\cY, R).
\]
Pointwise,
\[
    (f_*\shF)(y) \cong \colim_{f(x)\ge y} \shF(x).
\]
For the terminal map $p\colon \cX\to *$, the object $p_*\shF$ is the cosheaf of global cosections:
\[
    p_*\shF \cong \colim_{x\in \cX\op} \shF(x).
\]
\end{defn}

Because the projective model structure is objectwise, the functor $f_*$ is left Quillen and therefore admits a total left derived functor
\[
    \bbL f_*\colon D(\cX) \longrightarrow D(\cY).
\]
The key case is the terminal map, where derived pushforward is just homotopy colimit.

\begin{prop}\label{prop:derived-pushforward-hocolim}
For every $C\in D(\cX)$ and every $y\in \cY$ there is a natural identification in $D(\cY)$
\[
    (\bbL f_* C)(y) \cong \hocolim_{f(x)\ge y} C(x).
\]
In particular, for the terminal map $p\colon \cX\to *$ one has
\[
    \bbL p_* C \cong \hocolim_{\cX\op} C.
\]
\end{prop}

\begin{proof}[Proof]
The under-category defining the Kan extension is computed pointwise, and cofibrant replacement in the projective model structure is objectwise. Thus the derived left Kan extension is obtained by replacing the ordinary colimit with the corresponding homotopy colimit.
\end{proof}

Applying this to the constant cosheaf recovers the usual chain complex of the underlying space.

\begin{cor}\label{cor:constant-cosheaf-chains}
Let $p\colon \cX\to *$ be the terminal map. Then there is a natural isomorphism in $D(R)$
\[
    \bbL p_*(\underline{R}) \cong C_*^{\mathrm{sing}}(\cbB\cX;R).
\]
If $(X,\cX)$ is a regular CW complex, this further identifies with
\[
    \bbL p_*(\underline{R}) \cong C_*^{\mathrm{sing}}(X;R).
\]
\end{cor}

\begin{proof}[Proof]
One computes the homotopy colimit of the constant diagram by the usual bar construction. The resulting simplicial $R$-module is the simplicial chain complex of the nerve of $\cX$, hence its normalization is quasi-isomorphic to $C_*^{\mathrm{sing}}(\cbB\cX;R)$. For regular CW complexes, $\cbB\cX$ is the barycentric subdivision of $X$.
\end{proof}

\subsection{Assembly and representations}
From this point onward, fix a connected regular CW complex $(X,\cX)$ together with a chosen base $0$-cell. Also fix a universal covering
\[
    q\colon (\widetilde X,\widetilde{\cX}) \longrightarrow (X,\cX),
\]
and write $\pi:=\pi_1(X,x_0)$. Additional finiteness hypotheses will be imposed explicitly when they are needed.

We write
\[
    \Mod_R^\pi:=\Fun(\cB\pi,\Mod_R)
\]
for the category of $R$-linear representations of $\pi$. Equivalently, this is the category of left modules over the group ring $R[\pi]$. For $M\in \Mod_R$ we denote by
\[
    \Ind_\pi(M):=R[\pi]\otimes_R M
\]
the regular representation.

If $C\in \cM$ is a complex of cellular cosheaves, then its pullback $q^*C\in \Ch(\widetilde{\cX})$ carries a natural deck-transformation action of $\pi$. Passing to colimits or homotopy colimits over $\widetilde{\cX}\op$ therefore produces objects of $\Ch(R[\pi])$.

\begin{defn}
The \textbf{universal assembly} of a cellular cosheaf complex $C$ is
\[
    \Asm(C):=\colim_{\widetilde{\cX}\op} q^*C \in \Ch(R[\pi]).
\]
Its total left derived functor is denoted by
\[
    \bbL\Asm\colon D(\cX) \longrightarrow D(R[\pi]),
    \qquad
    \bbL\Asm(C)\cong \hocolim_{\widetilde{\cX}\op} q^*C.
\]
\end{defn}

This description is equivalent to the left-Kan-extension construction through the covering translation category; the universal cover formula is the one most convenient for the arguments used in the following sections.

\subsection{Global equivalences}

The assembly functor measures which objects become invisible after passage from local data on $X$ to $\pi$-representations.

\begin{defn}
A morphism $f\colon C\to D$ in $D(\cX)$ is called a \textbf{global equivalence} if the image
\[
    \bbL\Asm(f)\colon \bbL\Asm(C) \xrightarrow{\;\cong\;} \bbL\Asm(D)
\]
is an isomorphism in $D(R[\pi])$.
An object $C\in D(\cX)$ is \textbf{globally contractible} if
\[
    \bbL\Asm(C)\cong 0.
\]
We write
\[
    D_\gc(\cX):=\ker(\bbL\Asm)
\]
for the full subcategory of globally contractible objects.
\end{defn}
We will also say that morphism $f \in \cM$ is a global equivalence if its canonical image in $D(\cX)$ is. The first important fact is that elementary cosheaves already detect the basic behavior of assembly:

\begin{prop}\label{prop:assembly-elementary}
For every cell $x\in \cX$ and every $M\in \Mod_R$ there is a natural isomorphism
\[
    \bbL\Asm(\{\hx\}_M) \cong \Ind_\pi(M)
\]
in $D(R[\pi])$. In particular, if $x\le y$ in $\cX$, then the canonical morphism
\[
    \varphi_{xy}\colon \{\hx\}_M \longrightarrow \{\hy\}_M
\]
is a global equivalence.
\end{prop}

\begin{proof}[Proof]
Choose a lift $\widetilde x$ of $x$. The pullback $q^*\{\hx\}_M$ decomposes over the deck translates of $\widetilde x$, and each summand is supported on the opposite of the closure of a lift. Since indexing poset has the initial object $x$, its classifying space is contractible, hence the homotopy colimit over it is just the value at $x$, that is $M$. Summing over the $\pi$-orbit of lifts gives the regular representation $R[\pi]\otimes_R M$.

For $x\le y$, choose compatible lifts $\widetilde x\le \widetilde y$ in each deck-translated sheet. On each such sheet the map $q^*\varphi_{xy}$ induces the identity map on the corresponding copy of $M$, hence the induced map on assembly is an isomorphism.
\end{proof}

\begin{lem}\label{lem:global-equivalence-cofiber}
A morphism $f\colon C\to D$ in $D(\cX)$ is a global equivalence if and only if its cofiber belongs to $D_\gc(\cX)$.
\end{lem}

\begin{proof}
Since $\bbL\Asm$ is triangulated, it sends a distinguished triangle
\[
    C \longrightarrow D \longrightarrow \cofib(f) \longrightarrow \Sigma C
\]
to a distinguished triangle in $D(R[\pi])$. It follows that $\bbL\Asm(f)$ is an isomorphism precisely when the third term is zero, i.e. when $\bbL\Asm(\cofib(f))\cong 0$.
\end{proof}

The category $D_\gc(\cX)$ is the subcategory that will be inverted in the construction of homotopy local systems.

\subsection{Local systems}

A cellular sheaf or cosheaf is \textbf{locally constant} if every structure map is an isomorphism. For locally constant objects, passing from sheaves to cosheaves simply amounts to inverting the structure maps, so we may speak unambiguously of \textbf{local systems} on $\cX$.

The point is that a functor on $\cX$ is locally constant precisely when it factors through the groupoid completion of $\cX$.

\begin{prop}\label{prop:local-systems-groupoid}
Let $(X,\cX)$ be a connected regular CW complex. Then there are natural equivalences of categories
\[
    \Sh_\lc(\cX,R)
    \cong
    \CoSh_\lc(\cX,R)
    \cong
    \Fun(\Pi_1(X),R)
    \cong
    \Mod_R^\pi
    \cong
    \Mod_{R[\pi]}.
\]
Consequently,
\[
    D(\CoSh_\lc(\cX, R)) \cong D(R[\pi]).
\]
\end{prop}

\begin{proof}[Proof]
A locally constant functor sends every morphism in $\cX$ to an isomorphism; hence, it factors uniquely through the localization $\cX[\cX^{-1}]$. For any small category $\cA$, the groupoid completion $\cA[\cA^{-1}]$ is equivalent to the fundamental groupoid of $\cbB\cA$. Since $\cbB\cX\cong X$ and $X$ is connected, we obtain
\[
    \cX[\cX^{-1}] \cong \Pi_1(X) \cong \cB\pi.
\]
The remaining identifications are the standard equivalences between representations of $\pi$ and modules over $R[\pi]$, together with passage to derived categories.
\end{proof}

Thus, the target of assembly is the classical category in which locally constant data on $X$ are encoded by modules over the group ring. In the next section we describe the larger localization $D_\hlc(\cX)$ by similarly explicit algebraic means.

\section{Homotopy local systems}\label{sec:homotopy-local-systems}

We now use assembly to isolate the correct local objects. Assume for the rest of the discussion that $(X,\cX)$ is finite.

\subsection{Localization and fibrant objects}

\begin{defn}
A complex $C\in \cM$ is \textbf{homotopy locally constant} or a \textbf{homotopy local system} if for every relation $x\le y$ in $\cX$ the corresponding structure map
\[
    C(y) \longrightarrow C(x)
\]
is a quasi-isomorphism of chain complexes of $R$-modules.
\end{defn}

Thus a notion of a homotopy local system is obtained from the standard notion of a local system by replacing isomorphisms with quasi-isomorphisms. 

\begin{lem}\label{lem:hlc-global-zero}
Let $C\in D(\cX)$ be homotopy locally constant. If $\bbL\Asm(C)\cong 0$, then $C\cong 0$ in $D(\cX)$.
\end{lem}

\begin{proof}[Proof]
Pull $C$ back to the universal cover. Because every structure map of $C$ is a quasi-isomorphism, the induced functor
\[
    \widetilde{\cX}\op \longrightarrow D(R)
\]
inverts every morphism and therefore factors through the groupoid completion of $\widetilde{\cX}$. Since $\cbB\widetilde{\cX}\cong \widetilde X$ and $\widetilde X$ is simply connected, that groupoid completion is equivalent to a point. Hence $q^*C$ is constant up to isomorphism in $D(R)$; choose a lift $\widetilde x_0$ of the basepoint and write $M\cong C(\widetilde x_0)$ for the common value.

The derived assembly is then the homotopy colimit of a constant diagram,
\[
    \bbL\Asm(C)
    \cong
    C_*^{\mathrm{sing}}(\cbB\widetilde{\cX};R)\otimes_R^{\bbL} M.
\]
The augmentation $C_*^{\mathrm{sing}}(\cbB\widetilde{\cX};R)\to R$ and the $0$-simplex determined by $\widetilde x_0$ split $M$ off as a retract of the right-hand side. Therefore $\bbL\Asm(C)\cong 0$ implies $M\cong 0$. Since $C$ is homotopy locally constant, every value of $C$ is isomorphic to $M$ in $D(R)$, so $C\cong 0$.
\end{proof}

\begin{thm}[Characterization of the localized model structure]\label{thm:localized-model-structure}
Choose a set $\mathcal K_\gc$ of cofibrant complexes whose images generate $D_\gc(\cX)$ as a localizing subcategory, and let
\[
    S:=\{\,0\to K \mid K\in \mathcal K_\gc\,\}.
\]
Equivalently, one may express this formally as
\[
    S=\{\,0\to K \mid K\in D_\gc(\cX)\,\}.
\]
Then the left Bousfield localization
\[
    \cM_\hlc:=L_S\cM
\]
exists, see \cite[Chapter 4]{Hirschhorn02}. Its weak equivalences are exactly the global equivalences, and its fibrant objects are exactly the homotopy locally constant complexes. In particular, $\cM_\hlc$ is the left Bousfield localization of $\cM$ at global equivalences.
\end{thm}

\begin{proof}[Proof]
The category $\cM$ is a left proper combinatorial stable model category, so localization at the set $S$ exists \cite[2.11]{Barwick}.

An object $Z\in \cM$ is $S$-local precisely when
\[
    \uHom(K,Z)\cong 0
    \qquad\text{for every } K\in \mathcal K_\gc.
\]
Because $\mathcal K_\gc$ generates $D_\gc(\cX)$ as a localizing subcategory and $\uHom(-,Z)$ sends triangles to triangles and coproducts to products, this is equivalent to
\[
    \uHom(K,Z)\cong 0
    \qquad\text{for every } K\in D_\gc(\cX).
\]

Let $f\colon A\to B$ be a morphism in $\cM$, and write $C:=\cofib(f)$. For every $S$-local object $Z$, applying $\uHom(-,Z)$ to the distinguished triangle
\[
    A \longrightarrow B \longrightarrow C \longrightarrow \Sigma A
\]
shows that $\uHom(B,Z)\to \uHom(A,Z)$ is a quasi-isomorphism if and only if $\uHom(C,Z)\cong 0$. Hence $f$ is an $S$-equivalence if and only if $C$ is $S$-acyclic. In a stable left Bousfield localization, the $S$-acyclic objects form the smallest localizing subcategory containing the cofibers of the maps in $S$, so in the present case they are exactly the objects of $D_\gc(\cX)$. By Lemma~\ref{lem:global-equivalence-cofiber}, it follows that the weak equivalences in $\cM_\hlc$ are precisely global equivalences.

Since every object of the projective model structure on $\cM$ is fibrant, the fibrant objects of $\cM_\hlc$ are exactly the $S$-local objects. Suppose first that $Z$ is $S$-local. If $x\le y$ in $\cX$, then Proposition~\ref{prop:assembly-elementary} shows that the cofiber of
\[
    \varphi_{xy}\colon \{\hx\}_R \longrightarrow \{\hy\}_R
\]
belongs to $D_\gc(\cX)$. Therefore
\[
    \uHom(\cofib(\varphi_{xy}),Z)\cong 0,
\]
so the induced map
\[
    \uHom(\{\hy\}_R,Z) \longrightarrow \uHom(\{\hx\}_R,Z)
\]
is a quasi-isomorphism. Using Proposition~\ref{prop:elementary-adjunction}, this is exactly the structure map $Z(y)\to Z(x)$. Hence $Z$ is homotopy locally constant.

Conversely, let $C$ be homotopy locally constant and choose a fibrant replacement
\[
    u\colon C \longrightarrow LC
\]
in $\cM_\hlc$. By the first part of the proof, $u$ is a global equivalence, so $\cofib(u)$ belongs to $D_\gc(\cX)$. We claim that $\cofib(u)$ is again homotopy locally constant. Indeed, for each relation $x\le y$, evaluation on the distinguished triangle
\[
    C \longrightarrow LC \longrightarrow \cofib(u) \longrightarrow \Sigma C
\]
produces a morphism of distinguished triangles in $D(R)$, and since the vertical maps induced by $C$ and $LC$ are isomorphisms, the same holds for the vertical map induced by $\cofib(u)$. Thus $\cofib(u)$ is homotopy locally constant. Lemma~\ref{lem:hlc-global-zero} now implies that $\cofib(u)\cong 0$.

Finally, let $K\in D_\gc(\cX)$. Applying $\uHom(K,-)$ to the triangle above and using that $LC$ is $S$-local gives
\[
    \uHom(K,C)\cong 0.
\]
Hence $C$ is $S$-local. Therefore the fibrant objects of $\cM_\hlc$ are exactly the homotopy locally constant complexes.
\end{proof}

\begin{cor}\label{cor:homogeneous-envelope}
There is a functorial fibrant replacement
\[
    C \longmapsto V^\infty C
\]
in $\cM_\hlc$ such that $C\to V^\infty C$ is a global equivalence and $V^\infty C$ is homotopy locally constant.
\end{cor}

\begin{proof}
Every combinatorial model category admits functorial fibrant replacement \cite[2.5]{Barwick}. Applied to $\cM_\hlc$, this gives the claimed construction. In the surgery-theoretic literature $V^\infty C$ is the \emph{homogeneous envelope} of $C$ \cite[Ch.~10]{Ranicki92}, \cite[Def.~4.9]{RanickiWeiss}.
\end{proof}

\subsection{Quotients and generators}

Let
\[
    D_\hlc(\cX):=\Ho(\cM_\hlc)
\]
denote the homotopy category of the localized model structure. Also, recall the notion of a Verdier quotient of triangulated categories \cite[Rem. 2.1.9]{Neeman01}.

\begin{prop}\label{prop:verdier-quotient}
There is a natural equivalence of triangulated categories
\[
    D_\hlc(\cX) \cong D(\cX)/D_\gc(\cX).
\]
where the quotient on the right is the Verdier quotient.
\end{prop}

\begin{proof}
Stable left Bousfield localization identifies the localized homotopy category with the Verdier quotient by its acyclic objects. By Theorem~\ref{thm:localized-model-structure}, those acyclic objects are precisely the globally contractible complexes.
\end{proof}

This quotient description makes precise the slogan that homotopy local systems are obtained from all cellular cosheaf complexes by killing the part invisible to universal assembly.

\begin{prop}\label{prop:hlc-generator}
Let
\[
    \ell\colon D(\cX) \longrightarrow D_\hlc(\cX)
\]
be the localization functor. For any $0$-cell $x\in \cX$, the object $\ell(\{\hx\}_R)$ is a compact generator of $D_\hlc(\cX)$.
\end{prop}

\begin{proof}[Proof]
The elementary cosheaf $\{\hx\}_R$ is a bounded complex of finitely generated projective cellular cosheaves, hence compact in $D(\cX)$. The localization functor $\ell$ preserves coproducts and admits the fully faithful right adjoint given by inclusion of local objects, so $\ell(\{\hx\}_R)$ remains compact.

To see that it generates, recall that the elementary cosheaves generate $D(\cX)$. In the localized category every map $\varphi_{xy}$ with $x\le y$ becomes an isomorphism. Therefore if $v\le y$ with $v$ a $0$-cell, then $\ell(\{\widehat{v}\}_R)\cong \ell(\{\hy\}_R)$. Because $X$ is connected, any two $0$-cells are linked by a zig-zag of incident $1$-cells, so all objects $\ell(\{\widehat{v}\}_R)$ for $0$-cells $v$ are naturally isomorphic. Hence the localizing subcategory generated by $\ell(\{\hx\}_R)$ contains the image of every elementary cosheaf, and therefore all of $D_\hlc(\cX)$.
\end{proof}

\begin{proof}[Proof of Theorem B]
    It follows from Propostion \ref{prop:verdier-quotient} and Proposition \ref{prop:hlc-generator}.
\end{proof}

Proposition~\ref{prop:hlc-generator} reduces the algebraic side of the discussion to a single compact generator. To identify its derived endomorphism algebra, we therefore turn to a based-space construction in which loops start and end at the chosen basepoint.

\section{Loop DG algebra}\label{sec:loop-dg-algebra}

Return now to the finite case, and keep the notation
\[
    p:=\ev_1\colon \cE_x\cX \longrightarrow \cX.
\]
In Section~\ref{sec:cellular-loop-spaces} we constructed the missing geometric ingredient: a combinatorial loop object whose classifying space models $\Omega_xX$. We now show that the same construction controls the endomorphisms of the compact generator of $D_\hlc(\cX)$. To match the order of composition of endomorphisms, it is convenient to write formulas using the opposite poset $\varOmega_x\cX\op$. Since a category and its opposite have canonically homeomorphic classifying spaces, this changes only the bookkeeping of multiplication, not the underlying homotopy type.

\subsection{A canonical homotopy local system}

\begin{defn}
With $p=\ev_1$ as above, define the \textbf{canonical homotopy local system}
\[
    E_x:=\bbL p_*(\underline{R}) \in D(\cX),
\]
where $\underline{R}$ denotes the constant cellular cosheaf of $R$-modules on $\cE_x\cX$.
\end{defn}

\begin{prop}\label{prop:Ex-stalks}
For every cell $y\in \cX$ there is a natural isomorphism in $D(R)$
\[
    E_x(y) \cong C_*^{\mathrm{sing}}(\cbB\cP_{x,y}\cX;R).
\]
In particular,
\[
    E_x(x) \cong C_*^{\mathrm{sing}}(\cbB\varOmega_x\cX;R)
    \cong C_*^{\mathrm{sing}}(\cbB\varOmega_x\cX\op;R).
\]
\end{prop}

\begin{proof}
By Proposition~\ref{prop:derived-pushforward-hocolim},
\[
    E_x(y)
    \,\cong\,
    \hocolim_{p(z)\ge y} \underline{R}(z).
\]
The indexing poset is precisely the under-category $(y\downarrow p)$, so the homotopy colimit of the constant diagram is the singular chain complex of its classifying space:
\[
    E_x(y)
    \cong
    C_*^{\mathrm{sing}}(\cbB(y\downarrow p);R).
\]
By Proposition~\ref{prop:comma-retract}, $\cbB(y\downarrow p)\simeq \cbB\cP_{x,y}\cX$, and this gives the first claim. When $y=x$ we have $\cP_{x,x}=\varOmega_x\cX$. The final identification uses the canonical homeomorphism between the realizations of the nerve of a category and of its opposite.
\end{proof}

\begin{prop}\label{prop:Ex-is-hlc}
The object $E_x$ is homotopy locally constant.
\end{prop}

\begin{proof}
Let $y\le z$ in $\cX$. By Proposition~\ref{prop:Ex-stalks}, the corresponding structure map of $E_x$ is represented by the map on chains induced from the functor between under-categories
\[
    (z\downarrow p) \longrightarrow (y\downarrow p).
\]
After identifying these under-categories with $\cP_{x,z}\cX$ and $\cP_{x,y}\cX$ via Proposition~\ref{prop:comma-retract}, this becomes the map between path posets induced by the relation $y\le z$. Proposition~\ref{prop:path-posets-comparable} shows that the induced map on classifying spaces is a homotopy equivalence. Therefore, the associated map on singular chains is a quasi-isomorphism. Since this holds for every relation $y\le z$, the object $E_x$ is homotopy locally constant.
\end{proof}

Analogous to the homogeneous envelope $V^\infty\Gamma$ of \cite{RanickiWeiss}, the object $E_x$ is a fibrant replacement of the elementary representable cosheaf $\{\hx\}_R$ at the basepoint $x$:

\begin{prop}\label{prop:Ex-envelope}
There is a canonical morphism
\[
    \eta_x\colon \{\hx\}_R \longrightarrow E_x
\]
which becomes an isomorphism in $D_\hlc(\cX)$. Consequently $E_x$ is a homogeneous envelope of $\{\hx\}_R$.
\end{prop}

\begin{proof}[Proof]
Under Proposition~\ref{prop:Ex-stalks}, the value $E_x(x)$ is the chain complex of the loop space $\cbB\varOmega_x\cX$. The constant loop fence $(x)$ determines a canonical $0$-simplex of this space and hence a map
\[
    R \longrightarrow E_x(x).
\]
By Proposition~\ref{prop:elementary-adjunction}, this induces a morphism $\eta_x\colon \{\hx\}_R\to E_x$.

To see that $\eta_x$ becomes invertible after localization, let $L$ be any homotopy local system. Since $E_x$ is itself homotopy locally constant, it suffices to compare maps into such objects. Proposition~\ref{prop:elementary-adjunction} gives
\[
    \uHom_{D(\cX)}(\{\hx\}_R,L) \cong L(x).
\]
On the other hand, derived adjunction yields
\[
    \uHom_{D(\cX)}(E_x,L)
    \cong
    \uHom_{D(\cE_x\cX)}(\underline R,p^*L).
\]
The pullback $p^*L$ is homotopy locally constant on $\cE_x\cX$, and $\cbB\cE_x\cX$ is contractible by Proposition~\ref{prop:Ex-contractible}. By the same groupoid-completion argument used in Lemma~\ref{lem:hlc-global-zero}, $p^*L$ is therefore equivalent to the constant diagram with value $L(x)$. Hence
\[
    \uHom_{D(\cE_x\cX)}(\underline R,p^*L) \cong L(x).
\]
Under these identifications, precomposition with $\eta_x$ acts as the identity on $L(x)$. Therefore $\eta_x$ induces an isomorphism on derived mapping complexes into every local object, which is exactly the criterion for becoming an isomorphism in the left Bousfield localization. Thus $\eta_x$ identifies $E_x$ with the image of $\{\hx\}_R$ in $D_\hlc(\cX)$.
\end{proof}

\subsection{Derived endomorphisms}

\begin{prop}[Derived endomorphisms are controlled by cellular loops]\label{prop:derived-endomorphisms}
There is an isomorphism
\[
    \uHom_{D_\hlc(\cX)}(\{\hx\}_R,\{\hx\}_R)
    \cong
    C_*^{\mathrm{sing}}(\cbB\varOmega_x\cX\op;R)
\]
in $D(R)$.
\end{prop}

\begin{proof}
By Proposition~\ref{prop:Ex-envelope}, the objects $\{\hx\}_R$ and $E_x$ are isomorphic in $D_\hlc(\cX)$. Since $E_x$ is homotopy locally constant, it is already fibrant in the localized model structure, and therefore
\[
    \uHom_{D_\hlc(\cX)}(\{\hx\}_R,\{\hx\}_R)
    \cong
    \uHom_{D(\cX)}(\{\hx\}_R,E_x).
\]
Because $\{\hx\}_R$ is an elementary cofibrant object, Proposition~\ref{prop:elementary-adjunction} identifies the right-hand side with the value of $E_x$ at $x$:
\[
    \uHom_{D(\cX)}(\{\hx\}_R,E_x) \cong E_x(x).
\]
Now Proposition~\ref{prop:Ex-stalks} gives
\[
    E_x(x) \cong C_*^{\mathrm{sing}}(\cbB\varOmega_x\cX;R)
    \cong C_*^{\mathrm{sing}}(\cbB\varOmega_x\cX\op;R).
\]
Combining these identifications proves the claim, and also the first part of Theorem C.
\end{proof}

Let us abbreviate
\[
    A_x:=C_*^{\mathrm{sing}}(\cbB\varOmega_x\cX\op;R).
\]
By Corollary~\ref{cor:loop-dg-algebra}, together with the passage to the opposite multiplication convention, $A_x$ is a strictly associative DG algebra.

\subsection{DG algebra and augmentation}

\begin{prop}[Augmentation to the group ring]\label{prop:augmentation}
There is a morphism of DG algebras
\[
    \varepsilon\colon A_x \longrightarrow R[\pi]
\]
which induces an isomorphism on $H_0$.
\end{prop}

\begin{proof}
The target is regarded as a DG algebra concentrated in degree $0$. We define $\varepsilon$ on singular simplices. If
\[
    \sigma\colon \Delta^0 \longrightarrow \cbB\varOmega_x\cX\op
\]
is a $0$-simplex, then $\sigma$ determines a path component of $\cbB\varOmega_x\cX\op$. By Theorem~\ref{thm:loop-comparison}, path components are in bijection with
\[
    \pi_0\bigl(\cbB\varOmega_x\cX\op\bigr)
    \cong
    \pi_0(\Omega_x X)
    \cong
    \pi_1(X,x).
\]
We send $\sigma$ to the basis vector of $R[\pi]$ corresponding to that component. Every singular simplex of positive dimension is sent to $0$.

This is a chain map because the boundary of a $1$-simplex has endpoints lying in the same connected component, hence with the same image in the group ring, while boundaries of higher-dimensional simplices still land in positive degree and are therefore sent to $0$. It is multiplicative because the product on $A_x$ is induced by concatenation of loops: if $\sigma$ and $\tau$ are $0$-simplices, then
\[
    \varepsilon(\sigma\cdot \tau)
\]
is the basis element corresponding to the connected component of the concatenated loop, which is the product of the components of $\sigma$ and $\tau$ in $\pi_1(X,x)$. If either factor has positive degree, both sides are automatically zero.

Finally,
\[
    H_0(A_x)
    \cong
    H_0(\Omega_xX;R)
    \cong
    R[\pi_0(\Omega_xX)]
    \cong
    R[\pi_1(X,x)],
\]
so $H_0(\varepsilon)$ is an isomorphism.
\end{proof}

\begin{proof}[Proof of Theorem C]
    It follows from Proposition \ref{prop:derived-endomorphisms} and Proposition \ref{prop:augmentation}.
\end{proof}

\begin{cor}[Morita description of homotopy local systems]\label{cor:morita-description}
Let
\[
    A_x:=C_*^{\mathrm{sing}}(\cbB\varOmega_x\cX\op;R).
\]
Then the model category of homotopy local systems is Quillen equivalent to the model category of DG modules over $A_x$. Equivalently, there is a triangulated equivalence
\[
    D_\hlc(\cX) \cong D(\Mod_{A_x}).
\]
\end{cor}

\begin{proof}[Proof]
By Proposition~\ref{prop:hlc-generator}, the image of $\{\hx\}_R$ is a compact generator of $D_\hlc(\cX)$. Proposition~\ref{prop:derived-endomorphisms} identifies its derived endomorphism DG algebra with $A_x$. Standard derived Morita theory for algebraic stable model categories then upgrades this identification to a Quillen equivalence between the localized model category and DG modules over $A_x$, and hence to an equivalence of homotopy categories \cite[Theorem~3.1.1 and 3.3.3]{SchwedeShipley}.
\end{proof}

\end{document}